\documentclass[10pt, a4paper]{amsart}
\usepackage{amsmath,amstext,amscd, amssymb,txfonts, epsfig, psfrag, color, multicol, graphicx}
\usepackage[colorlinks=true, pdfstartview=FitV, linkcolor=blue, citecolor=blue, urlcolor=blue]{hyperref}

\theoremstyle{plain}

\newtheorem{thm}{Theorem}[section]
\newtheorem{lemma}[thm]{Lemma}

\newtheorem{cor}[thm]{Corollary}
\newtheorem{prop}[thm]{Proposition}

\theoremstyle{definition}

\newtheorem{remark}[thm]{Remark}

\newtheorem{Example}[thm]{Example}

\newtheorem{Open questions}[thm]{Open questions}
\newtheorem{Open question}[thm]{Open question}
\newtheorem{Open problems}[thm]{Open problems}
\newtheorem{Open problem}[thm]{Open problem}

\definecolor{dmagenta}{rgb}{.5,0,.5} 
\definecolor{dred}{rgb}{.5,0,0} 
\definecolor{green}{rgb}{0,.5,0} 
\definecolor{blue}{rgb}{0,0,0.5} 
\definecolor{black}{rgb}{0,0,0} 
\definecolor{vdgreen}{rgb}{0,.3,0} 
\definecolor{vdred}{rgb}{.3,0,0} 
\definecolor{red}{rgb}{1,0,0}

\def\bar{\overline}

\newcommand{\R}{\mathbb{R}}
\newcommand{\Z}{\mathbb{Z}}
\newcommand{\N}{\mathbb{N}}

\DeclareMathOperator{\Area}{Area}
\DeclareMathOperator{\degg}{deg}

\DeclareMathOperator{\Sol}{Sol}
\DeclareMathOperator{\GL}{GL}

\newcommand{\set}[1]{\left\{#1\right\}}

\newcommand{\abs}[1]{\left|#1\right|}
\renewcommand{\ni}{\noindent}

\newcommand{\ms}{\medskip}

\newcommand{\TB}[2]{\mbox{\tiny{$\left(\!\!\!\begin{array}{c}{#1} \\ {#2}\end{array}\!\!\!\right)$}}}

\begin{document}

\title{The Dehn function of Baumslag's Metabelian Group}

\author{M.~Kassabov and T.R.~Riley}

\date \today

\begin{abstract}
\ni
Baumslag's group is a finitely presented metabelian group with
a \mbox{$\Z \wr \Z$} subgroup. There is an analogue with
an additional torsion relation in which this subgroup
becomes $C_m \wr \Z$. We prove that Baumslag's group has
an exponential Dehn function. This contrasts with the torsion
analogues  which have  quadratic Dehn functions.
 \ms

\footnotesize{\ni \textbf{2000 Mathematics Subject Classification:  20F65, 20F10}  \\
\ni \emph{Key words and phrases:}
isoperimetric function, Dehn function, metabelian group,
Baumslag's group, lamplighter group}
\end{abstract}

\maketitle

\section{Introduction} \label{intro}

\emph{Baumslag's group}   $\Gamma$
is presented by
\begin{align*}
   \left\langle \, a,s,t \ \left| \   [a,a^t]=1, \ [s,t]=1, \ a^s=aa^t \,
\right. \right\rangle.
\end{align*}
[Our conventions are  $[x,y] = x^{-1}y^{-1} xy$ and
$x^{ny} = y^{-1} x^n y$ for group elements $x$, $y$ and integers $n$.]

Baumslag gave $\Gamma$ in~\cite{Baumslag} as the first example of a
finitely presented group with an abelian normal subgroup of infinite rank ---
namely,  the derived subgroup $[ \Gamma, \Gamma]$.
So $\Gamma$ is metabelian but  not polycyclic.
The subgroup $\langle a, t \rangle$ of $\Gamma$ is
$$
\Z \wr \Z  \ = \  \left( \bigoplus_{i \in \Z} \Z \right)  \rtimes  \Z  \ = \
\left\langle  \, a, t \ \left| \   \left[a,a^{t^k}\right]=1 \ (k \in \Z)  \,
 \right. \right\rangle.
$$
Introducing the relation $a^m=1$, where $m \geq 2$,
gives a   family  $\Gamma_m  =  \langle \, \Gamma  \mid a^m=1  \rangle$,
in which  the subgroup $\langle a, t \rangle$ is  $C_m \wr \Z$, where
   $C_m$ denotes the cyclic group of order $m$.

These groups appear in other guises: both $\Gamma$ and $\Gamma_m$
are groups of affine matrices and have Cayley graphs that are
horocyclic products of trees~\cite{BNW};
for $p$ prime,  $\Gamma_p$ is a  cocompact lattice in
$\Sol_5(\mathbb{F}_p(\!(t)\!))$~\cite{CT}.

\emph{Dehn functions} are invariants of finitely presentable groups
that are both geometric and  combinatorial in character.

The geometric perspective is that Dehn functions are analogous to classical
isoperimetric functions for simply connected Riemannian manifolds ---
these record  the  infimal $N$ such that any loop of length at most
$\ell$ can be spanned by a disc of area $N$.

The combinatorial perspective, our predominant point--of--view in this article,
is that  Dehn functions measure the complexity of a head--on attack on
the word problem. Suppose words $w = w(A)$ and $w'  = w'(A)$
represent the same element of a finitely presented group $\langle A \mid R \rangle$.
The \emph{cost} of converting $w$ to $w'$ is  the minimal $N$ such that
there is a sequence $w = w_0, \ldots, w_N = w'$ in which,  for $0 \leq i < N$,
there are words $u_i \alpha_i v_i$ and $u_{i+1} \beta_{i+1} v_{i+1}$
freely equal to $w_i$ and $w_{i+1}$, respectively,
such that $\alpha_{i} {\beta_{i+1}}^{-1} \in R^{\pm 1}$.
When  $w = w(A)$ represents the identity we define  $\Area(w)$ to be the cost of
converting $w$ to the empty word.  (So, when $w$ and $w'$ represent
the same group element, the \emph{cost} of converting  $w$ to  $w'$ is
$\Area(w^{-1}w')$.)  Equivalently, $\Area(w)$ in the minimal $N$ such that
$w$ freely equals $\prod_{i=1}^N {u_i}^{-1} r_i u_i$ for some words
$u_i=u_i(A)$ and some $r_i \in R^{\pm 1}$.

The  Dehn function $\Area : \N  \to \N$  of  $\langle A \mid R \rangle$ is defined by setting $\Area(n)$ to be  the maximum of  $\Area(w)$ over all words $w = w(A)$ that have length at most $n$ and
represent the identity.

For $f,g:\N \to \N$, we write $f \preceq g$ when there exists $C>0$
such that for all $n$, 
$$
f(n) \leq Cg(Cn+C)+Cn+C.
$$
This gives an equivalence relation capturing qualitative agreement of
growth rates:  $f \simeq g$ if and only if $f \preceq g$ and $g \preceq f$.
Any two finite presentations of the same group yield Dehn functions
that are $\simeq$--equivalent; indeed, up to $\simeq$, Dehn functions provide a quasi--isometry invariant for finitely presentable groups --- 
see, for example, \cite{Bridson6}.

Our main results are the following.
\begin{thm} \label{exp thm}
The Dehn function of  $\Gamma$ satisfies  $\Area(n) \simeq 2^n$.
\end{thm}
\begin{thm} \label{poly thm}
For all  $m$,  the Dehn functions of  $\Gamma_m$
satisfy  $\Area(n) \preceq n^4$.
\end{thm}

Our proofs of the upper bounds are via direct analysis of manipulations of words by relations.
The exponential lower bound in Theorem~\ref{exp thm} stems from a calculation of subgroup distortion in an extension.

A polynomial bound on  the Dehn functions of  $\Gamma_m$  in an earlier version
of this article spurred de~Cornulier and Tessera to prove in~\cite{CT}
that the Dehn function of $\Gamma_p$ is quadratic  for all prime $p$.
They view $\Gamma_p$, for $p$  prime, as a cocompact lattice in
$\Sol_5(\mathbb{F}_p(\!(t)\!))$, which they prove enjoys a
quadratic isoperimetric function by adapting  Gromov's proof from~\cite{Gromov}
of the corresponding result for $\Sol_5(\R)$.
They argue that their methods can be elaborated to cover  $\Gamma_m$  for all $m$.

Alternatively, as we are grateful to an anonymous referee for explaining, this quadratic bound can be obtained by combining results in~\cite{BNW} 
and~\cite{Drutu5}.   The description of a Cayley graph of $\Gamma_m$ in~\cite{BNW} as a horocyclic product is essentially the same as a horosphere corresponding to a barycentric ray in the product of the three trees. (It is contained in the horosphere and it is at finite Hausdorff distance from it.) Theorem~1.1, (1) in~\cite{Drutu5} gives a quadratic Dehn function for the horosphere, and so  for $\Gamma_m$.

We will give a brief account of our proof of Theorem~\ref{poly thm} --- it is elementary and it is interesting to see how the exponential upper bound for the Dehn function of $\Gamma$ improves in $\Gamma_m$.

Another strategy for establishing upper bounds on the Dehn functions
of $\Gamma_m$ for all $m$ has been suggested by N.~Brady~\cite{Brady_personal}.
With respect to a suitable finite presentation, the Cayley 2--complex
$\mathcal{C}$  of $\Gamma_m$ is the horocyclic product of  three
$(m+1)$--valent infinite trees~\cite{BNW} and so sits inside
a  $\textup{CAT}(0)$ space, namely the direct product of three such trees.
A loop in  $\mathcal{C}$ can be spanned by a disc in the ambient
$\textup{CAT}(0)$ space whose area is at most quadratic in the length of the loop.
Brady proposes pushing this disc into $\mathcal{C}$
in the manner of~\cite{ABDDY,Dison} to give a filling disc
with at most a quartic area.

\emph{Background.} Much is already known about the isoperimetry of solvable groups.
The Dehn functions of finitely generated  nilpotent groups 
admit
upper bounds of $\preceq n^{c+1}$ where $c$ is the nilpotency class~\cite{GHR, Gromov6},
and yet for all $c$ there are class $c$ examples with 
Dehn function $\simeq n^{c+1}$~\cite{BMS, Gersten, Pittet} 
and others with Dehn function $n^2$~\cite{Young4}.
There are also nilpotent examples with Dehn function 
$\simeq n^2 \log n $~\cite{Wenger}.  There are polycyclic groups such as lattices $\Sol_3(\R)$
with exponential Dehn function, and there are non--nilpotent polycyclic groups
such as higher--dimensional analogues of
$\Sol_{2m+1}(\R)$ that have quadratic Dehn functions~\cite{Drutu5, LP} for $m \geq 2$.
Venturing beyond polycyclic groups there are many metabelian groups with
Dehn function $\simeq 2^n$ --- the best known example is $\langle x, y \mid x^y = x^2 \rangle$.   The first non--polycyclic solvable group with Dehn function bounded above by a  polynomial were constructed in~\cite{AO} --- their Dehn functions grow at most cubically.

The groups $\Gamma$ and $\Gamma_m$ have received considerable attention in
other contexts.  The 3-dimensional integral homology group of $\Gamma$  is not finitely generated~\cite{BD}.
 In~\cite{GLSZ} it is shown that $\Gamma_2$ is a counterexample
to a strong version of the Atiyah Conjecture on  $L^2$--Betti numbers.
Random walks on Cayley graphs of  $\Gamma_m$   (Diestel--Leider graphs)
are studied in~\cite{BNW}.  In~\cite{CRwithcorrection}, $\Gamma_2$ was given as
the first known example of a finitely presented group with
\emph{unbounded dead--end depth} --- a property of the shapes of balls in the
Cayley graph.  In~\cite{Cleary} it is shown that
$\langle a, t \rangle \cong \Z \wr \Z$ is exponentially distorted inside $\Gamma$.

 \ms

\emph{The organisation of this article.}  Our exponential  lower bound on
the Dehn function of $\Gamma$ is established in Section~\ref{lower bound} and is proved to be sharp in Section~\ref{exp upper bound}.
In  Section~\ref{poly upper bound} we show how, for $\Gamma_m$,
the upper bound can be improved to quartic.

\ms

\emph{Acknowledgments.} We thank Sean Cleary for discussions which spurred this work, and an anonymous referee for a careful reading. The first author is grateful to  the NSF for partial support via grant DMS~0900932 and to both TIFR and Lufthansa for providing pleasant working environments.

\section{A Dehn function lower bound via extensions} \label{lower bound}

We will use the following general result. 

\begin{prop} \label{central extension lower bound}
Suppose $H$ is a normal subgroup of a group $G$ and  $A$ is a finite generating set for $G$.  Suppose  $\langle A \mid R \rangle$ is a finite presentation for  $G/H$.  Then each $r \in R$ can be regarded as representing an element of $H$. Let $B = \set{ r^g \mid r \in R, \, g \in G} \subseteq H$.

If a word $w=w(A)$ represents the identity in $G/H$, then
there exists a word $w'=w'(B)$ that equals $w$ in $G$
and has length equal to the area of $w$ in $\langle A \mid R \rangle$.
\end{prop}

\begin{proof}
As $w$ represents the identity in $G/H$, it freely equals $\prod_{i=1}^{N} {r_i}^{\varepsilon_i u_i}$
for some words $u_i =u_i(A)$, some $r_i \in {R}$ and
some ${\varepsilon_i}= \pm 1$, and $N=\Area(w)$.
But then  $w = \prod_{i=1}^{N} \left( {r_i}^{ u_i} \right)^{\varepsilon_i}$ in $G$ and as each ${r_i}^{u_i}$ represents an element of $B$, the result is proved.
\end{proof}

\begin{cor} \label{dist cor}
If $B$ is finite, then the distortion of $H$ in $G$ is a lower
bound for the Dehn function of $G$.
\end{cor}

This  was used in~\cite{BMS} in the case where $H$ is a  central subgroup of $G$ (and so $B=R$), as in the following example.

\begin{Example}
\label{Z2}
Here is an unusual proof that  the Dehn function
of $\Z^2$ is at least quadratic.
The 3-dimensional integral Heisenberg group
$$
\mathcal{H}_3  \ =  \
\langle  \, a, b, c \  \mid \ [a,b] = c, \ [a,c]=1, \ [b,c] =1 \, \rangle  \ =
\  \left(\begin{array}{ccc}1 & \Z & \Z \\0 & 1 & \Z \\0 & 0 & 1\end{array}\right)
$$
is  a central extension of
$\Z^2 = \langle  \, a, b, c   \,   \mid  \,  [a,b]=c, \  c=1 \, \rangle$
by $\Z = \langle c \rangle$.
The word $[a^n, b^n]$  represents $1$ in $\Z^2$ and $c^{n^2}$ in $\mathcal{H}_3$.
As $c^{n^2}$ has word length $n^2$ with respect to the generating set $B = \{c\}$,
the area of $[a^n, b^n]$ in
$ \langle  \, a,b, c \,   \mid  \,  [a,b]=c, \, c=1 \rangle$ is at least $n^2$.
\end{Example}

We are now ready to establish the exponential lower bound on the Dehn function of $\Gamma$, claimed in Theorem~\ref{exp thm}. We will apply Proposition~\ref{central extension lower bound} and Corollary~\ref{dist cor}  to
$$
\bar{\Gamma} \ = \ \left\langle \!    \ a, p, q, s, t \
\!\bigg|\!
\begin{array}{rlrlrlrlrlrlrlrl}
[a, a^t] \!\!\!\! & = p, & a a^t  \!\!\!\! & = a^{s}q,   & s^{-1} p s \!\!\!\! &  = p^{-1}, & t^{-1}p t \!\!\!\! &= p^{-1},  & [a,p]  \!\!\!\!  &   =1   \\
\ [s,t] \!\!\!\!  & =1, &  [p,q] \!\!\!\! & =1, &  s^{-1} q s \!\!\!\! &  = q^{-1}, & t^{-1}q t \!\!\!\! &= q^{-1},  & [a,q] \!\!\!\! & = 1
\end{array} \!\!
\right\rangle
$$
and its normal subgroup $H := \langle p, q \rangle$.   Baumslag's group $\Gamma$ is $\bar{\Gamma} / H$.

In this case,  $H$ is not central but  Corollary~\ref{dist cor} applies none--the--less because, if $g \in \bar{\Gamma}$ and $h \in H$ then either
 $h^g = h$ or $h^g = h^{-1}$, and so $B$ is finite.   The next two lemmas identify $H$ and  show it is exponentially distorted in $\bar{\Gamma}$.


\begin{lemma} \label{almost central extension}
$H$ is isomorphic to $\Z^2$.
\end{lemma}

\begin{proof}
Define  $R$ to be the ring $\Z[x,x^{-1},(x+1)^{-1}]$ and
$$
A = \left( \begin{array}{ccc}
1 &  1 &  0 \\   & 1 &  1  \\   &   & 1
\end{array} \right),
\ \  \ \
P = \left( \begin{array}{ccc}
1 &  0 &  -2x-1 \\   & 1 &  0  \\   &   & 1
\end{array} \right),
\ \ \ \
Q =   \left( \begin{array}{ccc}
1 &  0 &  -x-1 \\   & 1 &  0  \\   &   & 1
\end{array} \right),
$$
$$
S = \left( \begin{array}{ccc}
1 & 0 & 0 \\   &  x+1  & 0  \\   &   & -x^2-x
\end{array} \right),
\ \  \ \
T = \left( \begin{array}{ccc}
1 &  0 &  0 \\   &  x &  0  \\   &     &  -x^2-x
\end{array}    \right).
$$
Consider the analogues of the defining relations of $\bar{\Gamma}$.
In $\GL_3(R)$,
$$
[A, A^T] = P, \ \  \ A A^T  = A^{S}Q, \ \ \  [S,T]  = 1, \  \  \
[P,Q]=1, \ \ \ [A,P]=1,  \ \ \ [A,Q]=1
$$
hold, but
\begin{equation} \label{failed relations}
S^{-1}PS=P^{-1}, \ \ \ S^{-1}QS=Q^{-1}, \ \ \
T^{-1}PT=P^{-1}, \ \ \ T^{-1}QT=Q^{-1}
\end{equation}
fail since, for $f \in R$,
\begin{equation} \label{not central}
S^{-1}
\left( \begin{array}{ccc}
1 &  0 &  f \\  &  1 &  0  \\   &     &  1 \end{array}
\right)
S
\  = \
T^{-1}
\left( \begin{array}{ccc}
1 &  0 &  f \\  &  1 &  0  \\   &     &  1 \end{array}
\right)
T
\  = \
\left( \begin{array}{ccc}
1 &  0 &  -(x^2+x) f  \\  &  1 &  0  \\   &     &  1
\end{array}
\right).
\end{equation}

Let $R^*$ denote the invertible elements of $R$.
The matrices $A,P,Q,S,T$ are in
$$
G \ :=  \
\left( \begin{array}{ccc}
1 & R & R \\    & R^* & R \\   &   & R^*
\end{array}\right)  \ \leq \ \GL_3(R).$$
Let $\tau = (-1 + \sqrt{5})/2$.  Applying the ring homomorphism $R \to \Z[\tau]$,
defined by $f(x) \mapsto f(\tau)$, to the top--right entry maps $G$
to the group
$$
\hat{G} \ :=  \
\left( \begin{array}{ccc}
1 & R & \Z[\tau] \\    & R^* & R \\   &   & R^*
\end{array}\right).$$
[Lifting to $G$ and multiplying there
 leads to the group operation on $\hat{G}$.]

The calculation~\eqref{not central} shows that in $\hat{G}$ the (images of the)
relations~\eqref{failed relations} hold  since $\tau^2 + \tau=1$.
So mapping
$a \mapsto A$, $p \mapsto P$, $q \mapsto Q$, $s \mapsto S$, $t \mapsto T$
induces a homomorphism $\bar{\Gamma} \to \hat{G}$.  (Incidentally, it is possible to show that this is an injection, but that is not a result we need here.)
The images of  $P$ and $Q$ inside $\hat{G}$  are the matrices
that differ from the identity only in that they have $-\sqrt{5}$ and $(-1-\sqrt{5})/2$,
respectively, as their top--right entries, and so they generate
$$ \left( \begin{array}{ccc}
1 & 0 & \Z[\tau] \\    & 1 & 0 \\   &   & 1
\end{array}\right).$$  It follows that $H \cong \Z^2$.
\end{proof}




Let $F_n$ denote the $n$--th Fibonacci number, defined by $F_0 =0$, $F_1 =1$ and $F_{n+2} = F_{n+1} + F_n$ for   $n \geq 0$.

\begin{lemma} \label{distortion calculation}
 $\left[a, a^{t^n}\right] =   p^{(-1)^{n+1}F_n}$ in $\bar{\Gamma}$ for  all  $n \geq 0$.
\end{lemma}

\begin{proof}
The commutator  $\left[a, a^{t^n}\right]$ represents the identity in $\Gamma$ --- a fact whose proof we postpone to Lemma~\ref{Cn upper} ---
and so, by the previous lemma, represents an element $p^{\lambda} q^{\mu}$ of $H$ in $\bar{\Gamma}$.  We can find $\lambda$ and $\mu$ using the matrices from our proof of of Lemma~\ref{almost central extension}.  We calculate that in $G$,
$$
\left[ A, A^{T^n} \right] \ = \
\left( \begin{array}{ccc}
1 & 0 &  (-x-1)^n -x^n  \\    & 1 &  0 \\   &   & 1
\end{array}\right),
$$
which has image
$$
\left( \begin{array}{ccc}
1 & 0 &  (-1)^n \sqrt{5} F_n \\    & 1 &  0 \\   &   & 1
\end{array}\right)
$$
in  $\hat{G}$.
Now, the matrix in $\hat{G}$ corresponding to $p^{\lambda} q^{\mu}$ has $\lambda(-\sqrt{5}) + \mu (-1-\sqrt{5})/2$ in the upper right corner, and otherwise agrees with the identity matrix.   This entry equals
$(-1)^n \sqrt{5} F_n$ precisely when $\lambda= (-1)^{n+1}F_n$ and $\mu=0$.
\end{proof}

We conclude that  the Dehn function of $\Gamma$ grows $\succeq 2^n$ by Corollary~\ref{dist cor}.


\section{An exponential upper bound on the Dehn function of $\Gamma$} 
\label{exp upper bound}

Throughout this section  we calculate using the presentation
$$
\left\langle \, a,s,t \ \left| \   [a,a^t]=1, \ [s,t]=1, \ a^s=aa^t \,
\right. \right\rangle
$$ for
 $\Gamma$  introduced in Section~\ref{intro}.

We begin by focussing on a particular family of words.
For $n \in \Z$,  define
$$
C(n) \  :=  \ \Area \, \left[a, a^{t^n}\right],
$$
the minimum cost  to  convert $\left[a, a^{t^n} \right]$ to the empty word,
or equivalently $a a^{t^n}$ to $a^{t^n} a$, in $\Gamma$.
That these commutators represent the identity is the key point in
Baumslag's proof in~\cite{Baumslag}  that $\Gamma$ is metabelian.

 \begin{lemma}  \label{Cn upper}
 $\left[a, a^{t^n}\right]=1$ in $\Gamma$ and
 $C(n) \ \leq \ 4^n$ for all $n \geq 1$.
\end{lemma}

\begin{proof}
We induct on $n$.  The cases $n=1, 2, 3$ can be checked directly.
For the induction step, assume $n \geq 3$.  We have $aa^{t^n}  =a^{t^n} a$
at a cost of at most $C(n)$.
So $(aa^{t^n})^s  = (a^{t^n} a)^s$ and therefore
$aa^t (a^t a)^{t^n} =  (a^t a)^{t^n}  a a^t$ at an additional cost of $6 + 4n$.
Thus we have  $aa^t  a^{t^{n+1}} a^{t^n} =  a^{t^{n+1}} a^{t^n} a a^t$.
Then, using the commutator relations $\left[a,a^{t^{n-1}}\right]=1$ and
$\left[a,a^{t^n}\right]=1$ once and twice, respectively,
at a cost of $C(n-1) + 2C(n)$, gives
$a  a^{t^{n+1}} a^t a^{t^n} =  a^{t^{n+1}}  a a^t a^{t^n}$ and therefore
$aa^{t^{n+1}}  =a^{t^{n+1}} a$.
Thus $C(n+1) \leq 3C(n) + C(n-1) + 4n + 6 \leq 3 \cdot 4^n + 4^{n-1} +4n+6$
which is at most $4^{n+1}$ when $n \geq 3$.
\end{proof}

The  notation and techniques used in our next lemma are similar to
those in Section~4.3 of~\cite{GKKL}, which in turn draws on~\cite{CRW}.
It concerns the impact of the relations $a^s = a a^t$ and $a^s = a^t a$
in $\Gamma$.  For a  polynomial $f(x) = \sum_{i=0}^n c_i x^i$ in $\Z[x]$
and letters $a$ and $r$, define the word
\begin{align*}
[\![a]\!]_r^f  & \ := \ a^{c_0}  r^{-1} a^{c_1} r^{-1}    \ldots  a^{c_n} r^n,
\end{align*}
which freely equals $a^{c_0}   a^{c_1r}   \ldots   a^{c_n r^n}$.

\begin{lemma}  \label{x + 1 lemma}
For all polynomials $f(x) = \sum_{i=0}^n c_i x^i$ satisfying
$\max_i \abs{c_i} \leq c$,
$$[\![a]\!]_s^{f(x)}  \ = \  [\![a]\!]_t^{f(x+1)}$$
in   $\Gamma$  and the cost of the equality is at most $D_c(n):=   c^2 4^{n+1}$.
\end{lemma}

\begin{proof}  The lemma is trivial in the case $c=0$,
so we can assume that $c \geq 1$.
We induct on  $n$.   When $n=0$ the words are identical.
For the induction step assume $n \geq 1$.
Let $\hat{f}(x) = \sum_{i=1}^n c_i x^{i-1}$, so that
$f(x) = c_0 + x \hat{f}(x)$.
$$
\begin{array}{lllllll}
[\![a]\!]_s^{f(x)} &  \stackrel{\textup{I}}{=} & a^{c_0} \left( [\![a]\!]_s^{\hat{f}(x)} \right)^s
 & \stackrel{\textup{II}}{=} & a^{c_0} \left( [\![a]\!]_t^{\hat{f}(x+1)} \right)^s
 & \stackrel{\textup{III}}{=} & a^{c_0}  [\![a^s]\!]_t^{\hat{f}(x+1)}   \\
 & \stackrel{\textup{IV}}{=} & a^{c_0} [\![aa^t]\!]_t^{\hat{f}(x+1)}
  & \stackrel{\textup{V}}{=} & a^{c_0} [\![a]\!]_t^{(x+1)\hat{f}(x+1)}
  & \stackrel{\textup{VI}}{=} &  [\![a]\!]_t^{f(x+1)},
  \end{array}
$$
where (I) is a free equality, (II) costs at most $D_c(n-1)$  by the induction hypothesis,
(III) costs at most $2n$ applications of $[s,t] =1$,
(IV) uses $a^s = aa^t$ and costs the sum of the absolute value of the coefficients of $\hat{f}(x+1)$ and so at most $c 2^n$, (V) will be explained momentarily,
and (VI) is a free equality.  For each coefficient $m$ of  $\hat{f}(x+1)$,
equality (V) uses $[a,a^t]=1$ to transform a word $(aa^t)^{m}$ to $a^{m} a^{m t}$
at a cost of no more than $m^2$.  So a crude upper bound for the total cost of (V)
is $c^2$ times the square of the sum of the coefficients of
$\sum_{i=1}^n (x+1)^{i-1}$ --- that is,
$c^2 \left( \sum_{i=1}^n 2^{i-1} \right)^2 \leq  c^2 4^n$.
Thus $[\![a]\!]_s^f$ can be transformed to $[\![a]\!]_t^{f(x+1)}$
at a total cost of at most
\begin{equation} \label{summing stuff}
D_c(n-1)  + 2n + c 2^n+  c^2 4^n   \ \leq \    D_c(n).
\qedhere
\end{equation}
\end{proof}

The following proposition combines with Lemma~\ref{Cn upper} to prove
that the Dehn function of $\Gamma$ admits an exponential upper bound, and
so completes our proof of Theorem~\ref{exp thm}.

\begin{prop} \label{Reduction}
There exists $K>0$ such that for all $n \geq 1$, the Dehn function of $\Gamma$
satisfies
$$
\Area(n) \ \leq \    K^n \, \max \set{ C(i) \mid 0 \leq i \leq 6n}.
$$
\end{prop}

\begin{proof}
Suppose a word $w = w(a,s,t)$ has length $n$ and represents $1$  in $\Gamma$.
Let $$\rho : \Gamma \to \Z^2  \ = \   \langle s,t \rangle$$  be the retract
arising from killing $a$.   We may assume $w$ contains at least one letter
$a^{\pm 1}$ for otherwise $w$ represents $1$ in  $\Z^2 = \langle s,t \rangle$
and so $\Area(w) \leq n^2$.     We will convert $w$ to successive words
$w_1, \cdots, w_5$ and then to the empty word, and will then sum the costs.

After each prefix $u$ of $w$ that ends in a letter $a^{\pm 1}$, insert the word
$s^{\alpha} t^{\beta} t^{-\beta} s^{-\alpha} $ where $t^{-\beta} s^{-\alpha}$
equals $\rho(u)$ in $\Z^2$.  The resulting word $w_1$ freely equals $w$.
Note that $\abs{\alpha} + \abs{\beta} \leq n$.

At a cost of at most $2n^3$, use the relation $[s,t]=1$ to change $w_1$ to a word
$$
w_2  \ = \  \prod_{i=1}^k  a^{\epsilon_i s^{\alpha_i} t^{\beta_i}}
$$
in which each $\epsilon_i = \pm 1$, $\abs{\alpha_i}  + \abs{\beta_i} \leq n$ and
$k \leq n$ ---  in front of each   $a^{\pm 1}$ in $w_1$, there is a subword on
$s^{\pm 1}$ and $t^{\pm 1}$ of length at most $2n$ to convert  to
$t^{-\beta_i} s^{-\alpha_i}$ for some $\alpha_i$ and $\beta_i$;
each of the  at most $n$  such subwords  costs at most $2n^2$.

Next we would like to apply Lemma~\ref{x + 1 lemma} with $f(x) = x^{\alpha_i}$
to each $a^{s^{\alpha_i}}$ to  reach a word that is a product of conjugates of
$a^{\pm 1}$ by powers of $t$.   But this presupposes that each $\alpha_i$
is non--negative, which may not be the case.  We work around this issue as follows.
Define $\alpha :=   \min_i \alpha_i$.  Then $0 \leq \alpha_i - \alpha \leq 2n$
for all $i$.  Instead of continuing to transform $w_2$, we will work with
${w_2}^{s^{-\alpha}}$.   We may do this because in any finitely presented
group the area of a  word $v$ representing $1$ is the same as that of $v^u$
for any word $u$.

It costs at most  $2n^2$  to convert
$a^{\epsilon_i s^{\alpha_i} t^{\beta_i}  s^{-\alpha}}$  to
$a^{\epsilon_i s^{\alpha_i-\alpha} t^{\beta_i}}$
since $\abs{\alpha}, \abs{\beta_i} \leq n$.
So, as $k \leq n$, at most $2n^3$ relations are required to convert
${w_2}^{s^{-\alpha}}$  to
$$
w_3 \ := \ \prod_{i=1}^k  a^{\epsilon_i s^{\alpha_i - \alpha} t^{\beta_i}}.
$$

We can now apply Lemma~\ref{x + 1 lemma} since  each ${\alpha_i - \alpha} \geq 0$.
It converts each $a^{s^{j}}$  in $w_3$  to $[\![a]\!]_t^{(1+x)^j}$, and so
yields
$$
w_4 \ := \
\prod_{i=1}^k  \left( [\![a]\!]_t^{(1+x)^{\alpha_i - \alpha}}   \right)^{\epsilon_i  t^{\beta_i}}
$$
at a total cost of at most   $nD_1(\alpha_i - \alpha) \leq nD_1(2n)$.
But then $w_4$ freely equals a word
$$
w_5 \ =  \ \prod_{i=1}^{l}  {a}^{\mu_i t^{\gamma_i}}
$$
for some integers $l, \mu_i, \gamma_i$ satisfying the following
(crude) inequalities
\begin{align*}
l & \ \leq \  (2n+1)n, \\
\abs{\mu_i} & \ \leq \  \max \set{ \ \left.\TB{j}{i} \ \right| \  1 \leq j \leq 2n, \  0 \leq i \leq j \  } \ \leq \ 2^{2n}, \\
\abs{\gamma_i} & \ \leq \  \max_i (\alpha_i - \alpha + \beta_i) \ \leq  \ 3n.
\end{align*}

Now, $w_5$ represents the identity in the subgroup
$$
\langle a, t \rangle \  = \  \Z \wr \Z  \   =  \
\left\langle  \, a, t \ \bigg| \  \left[a,a^{t^k}\right]=1 \ (k \in \Z)  \,
\right\rangle
$$
and is freely equal to a product of at most $2^{2n}(2n+1)n$  terms of
the form $a^{\pm t^j}$ in which $\abs{j} \leq 3n$.
So reordering these terms so as to collect all those in which the power of $t$
agree will give a word which freely reduces to the empty word.
This reordering can be achieved at a cost of  no more than
$\left( 2^{2n}(2n+1)n \right)^2$ commutators  $\left[a,a^{t^i}\right]$ in which
$0 \leq i \leq 6n$.

Summing our cost estimates, we have
$$
\Area(n) \ \leq \    2n^3 + 2n^3+ nD_1(2n) +
         \left( 2^{2n}(2n+1)n \right)^2 \, \max \set{ C(i) \mid 0 \leq i \leq 6n},
$$
and the result follows.
\end{proof}

\section{A quartic  upper bound on the Dehn function of $\Gamma_m$} \label{poly upper bound}

In this section we calculate using this presentation for $\Gamma_m$:
$$
\left\langle \, a,s,t \ \bigg| \   [a,a^t]=1, \ [s,t]=1, \ a^s=aa^t, \ a^m=1 \,
\right\rangle.
$$

The disparity between  $\Gamma_m$ and $\Gamma$ first appears in the following
analogue of Lemma~\ref{x + 1 lemma}.   Adapting the notation of Lemma~\ref{x + 1 lemma},
for a polynomial  $f(x) = \sum_{i=0}^n \hat{c}_ix^i $ in   $\Z[x]$
and letters $a$ and $r$, define words
\begin{align*}
\{\!\{a\}\!\}_r^f  & \ =  \ a^{c_0}  r^{-1}  a^{c_1}r^{-1}   \ldots  a^{c_{n-1}}r^{-1}  a^{c_n} r^n \\
{}^f\{\!\{a\}\!\}_r & \  = \  r^{-n} a^{c_n} r a^{c_{n-1}}    \ldots  r  a^{c_1}r   a^{c_0}
\end{align*}
where $c_i \in \set{0, \ldots, m-1}$ equals $\hat{c}_i$ modulo $m$.  Calculating as for  Lemma~\ref{x + 1 lemma} we find the exponential cost estimates do not appear since the relation $a^m=1$ can be used to keep the exponents of $a$ between $0$ and $m-1$, and  we get ---

\begin{lemma}  \label{better x + 1 lemma}
$\{\!\{a\}\!\}_s^{f(x)} = \{\!\{a\}\!\}_t^{f(x+1)}$ and
${}^{f(x)}\{\!\{a\}\!\}_s = {}^{f(x+1)}\{\!\{a\}\!\}_t$ in   $\Gamma_m$.
In each case, the cost of converting one word to the other in
$\Gamma_m$ is less than $K_m(n) :=  10m^2 n^2 +10$.
\end{lemma}


For $f,g \in \Z[x]$, define
$$
\begin{array}{lll}
\sigma_{f,g}  & := \  \{\!\{a\}\!\}_s^f \  \{\!\{a\}\!\}_s^g \  \left(\{\!\{a\}\!\}_s^{f+g}\right)^{-1} \\ 
\tau_{f,g} & := \  \{\!\{a\}\!\}_t^f \  \{\!\{a\}\!\}_t^g \  \left(\{\!\{a\}\!\}_t^{f+g}\right)^{-1}. 
\end{array}
$$

It follows from Lemma~\ref{Cn upper}  that $\tau_{f,g}$ represents the identity in $\Gamma$.
Our next two lemmas are preparation for  estimating $\Area(\tau_{f,g})$.  We omit the proof of the first.

\begin{lemma} \label{small change}
For $n \geq  \max \set{ \degg f, \degg g}$,
\begin{align*}
\Area (\tau_{f,g}) & \ \leq \  \Area(\tau_{f \pm 1,g}) +2  \\
\Area (\tau_{f,g}) & \ \leq \    \Area(\tau_{f ,g \pm x^n}) +2.
\end{align*}
\end{lemma}


\begin{lemma} \label{reduce degree}
For  $n := 1+ \max \set{ \degg f, \degg g}$,
\begin{align*}
\Area \left( \tau_{(x+1)f,(x+1)g} \right) & \ \leq \
\Area \left( \tau_{f,g} \right) +  6K_m(n).
\end{align*}
\end{lemma}

\begin{proof} In $\Gamma_m$,
$$
\tau_{(x+1)f, (x+1)g} \ = \ \sigma_{xf(x-1), xg(x-1)} \ = \
\left( \sigma_{f(x-1), g(x-1)} \right)^t  \ = \  \left(\tau_{f, g} \right)^t,
$$
where the second equality is free, and the first and third
stem from  applications of Lemma~\ref{better x + 1 lemma}.
As $\tau_{f,g}$ and $\left( \tau_{f,g}  \right)^t$ have the same area,
the result is established.
\end{proof}

\begin{prop} \label{adding two polys}
Suppose $f,g$ are  polynomials of degree at most $n$. Then
$\tau_{f,g}$ represents the identity in $\Gamma_m$ and
$$
\Area\left(\tau_{f,g}\right) \ \leq \  6nK_m(n) + 4mn.
$$
\end{prop}

\begin{proof} Induct on $n$.   If $n=0$, then $\tau_{f,g}$ freely equals the empty word.
For the induction step assume $n \geq 1$ and $f$ has degree at least one.
(Otherwise,  $\tau_{f,g}$ freely equals the empty word and the result is immediate.)
Add a suitable constant to $f$ and  a suitable constant times $x^n$ to $g$
to obtain polynomials $f_0 = (x+1)f_1$ and $g_0 = (x+1)g_1$  both divisible by $x+1$.
By Lemmas~\ref{small change} and~\ref{reduce degree},
$$
\Area (\tau_{f,g}) \ = \  \Area(\tau_{f_0,g_0}) + 4m \ \leq \
\Area(\tau_{f_1,g_1}) + 6K_m(n) + 4m.
$$
So, by induction,
$$
\Area (\tau_{f,g}) \ \leq \   6(n-1)K_m(n-1) + 4m(n-1) + 6K_m(n) +4m \ \leq \  6nK_m(n) +4mn.  \vspace*{-5mm}
$$
\end{proof}

We are now ready to prove Theorem~\ref{poly thm}.
Suppose a word $w = w(a,s,t)$ has length $n$ and represents $1$  in $\Gamma_m$.
At a cost of at most $2n^3$, proceed to a word
$$
w_2 \ = \ \prod_{i=1}^k  a^{\epsilon_i s^{\alpha_i } t^{\beta_i}}
$$
for which $\epsilon_i = \pm 1$, and $\abs{\alpha_i} + \abs{\beta_i} \leq n$,
and $k \leq n$ exactly as in our proof of Proposition~\ref{Reduction}.
Define $\alpha : = \min_i \alpha_i$ and  $\beta : = \min_i \beta_i$.
Then $w_2$ has the same area as ${w_2}^{s^{-\alpha} t^{-\beta}}$.
It costs at most $2n^3$ applications of $[s,t]=1$ to convert
${w_2}^{s^{-\alpha} t^{-\beta}}$ to
$$
w'_3 \ := \ \prod_{i=1}^k  a^{\epsilon_i s^{\alpha_i - \alpha } t^{\beta_i - \beta}}.
$$

Let $f_i(x) :=x^{\alpha_i - \alpha}$.  Then $w'_3$ freely equals
$
\prod_{i=1}^k  \left( \{\!\{a\}\!\}^{f_i}_s \right)^{\epsilon_i  t^{\beta_i-\beta}}
$
which, by Lemma~\ref{better x + 1 lemma},
becomes
$$
w'_4 \ := \ \prod_{i=1}^k  \left( \{\!\{a\}\!\}^{f_i(x+1)}_t \right)^{\epsilon_i  t^{\beta_i-\beta}}
$$
at a cost of at most  $nK_m(n)$.
By applying Proposition~\ref{adding two polys} in the case $g=-f$,
we can change each $\left( \{\!\{a\}\!\}^{f_i(x+1)}_t \right)^{\epsilon_i}$
for which $\epsilon_i = -1$ to  $ \{\!\{a\}\!\}^{-f_i(x+1)}_t$
at a total cost of at most $6n^2K_m(n)+4mn^2$.  We then have a word
freely equal to
$$
w'_5 \ := \ \prod_{i=1}^k   \{\!\{a\}\!\}^{g_i}_t
$$
where $g_i(x) : = \epsilon_i x^{\beta_i} f_i(x+1)$.
For $1 \leq i < k$, Proposition~\ref{adding two polys} establishes the equality
$$
\{\!\{a\}\!\}^{\sum_{i=1}^j g_i}_t \prod_{i=j+1}^k   \{\!\{a\}\!\}^{g_i}_t \  = \
\{\!\{a\}\!\}^{\sum_{i=1}^{j+1} g_i}_t \prod_{i=j+2}^k   \{\!\{a\}\!\}^{g_i}_t
$$
at a cost of  $6nK_m(n) +4mn$.  So $w'_5$ can be transformed to
$w'_6 :=   \{\!\{a\}\!\}^{\sum_{i=1}^k g_i}_t$ at a cost of at most $6n^2K_m(n) + 4mn^2$.
Since no letters $s^{\pm 1}$ occur in $w'_6$, it represents the identity in
$$
C_m \wr \Z   \ = \
\left\langle  \, a, t \ \bigg|
\ a^m=1,  \left[a,a^{t^k}\right]=1 \ (k \in \Z)  \, \right\rangle
$$
and so $\sum_{i=1}^k g_i = 0$ and $w'_6$ is in fact the empty word.

Theorem~\ref{poly thm} then follows from summing the cost estimates:
$$
\Area(n) \ \leq \   4n^3 + nK_m(n) +  12n^2K_m(n) + 8mn^2.
$$

\begin{remark}
The Dehn function of $\Gamma_m$ is bounded below by $n^2$.  (Use the argument in Example~\ref{Z2} or deduce this from the fact that $\Gamma_m$ is not hyperbolic since it has a $\Z^2$ subgroup.)
It seems that one might be able to  do  the estimates above more carefully
and use essentially the method to obtain a cubic upper bound for the Dehn function.
However, we would be surprised if these types of combinatorial arguments could lead to
the quadratic upper bound obtained  via geometric methods.
\end{remark}

\bibliographystyle{plain}
\bibliography{$HOME/Dropbox/Bibliographies/bibli}

\small{
\ni  \textsc{Martin Kassabov} \rule{0mm}{6mm} \\
Department of Mathematics,
Cornell University, 310 Malott Hall, Ithaca, NY 14850, USA \\ 
\textsc{Current Address}\\
School of Mathematics,
University of Southampton, Highfield, Southampton, SO17 1BJ, UK\\ \texttt{kassabov@math.cornell.edu},
\texttt{martin.kassabov@southampton.ac.uk}
}

\small{
\ni  \textsc{Timothy R.\ Riley} \rule{0mm}{6mm} \\
Department of Mathematics,
Cornell University, 310 Malott Hall, Ithaca, NY 14850, USA \\ \texttt{tim.riley@math.cornell.edu}, \
\href{http://www.math.cornell.edu/~riley/}{http://www.math.cornell.edu/$\sim$riley/}
}

\end{document}